\documentclass[a4paper,final, 10pt]{amsart}

\usepackage{tempstyle}
\usepackage[obeyFinal]{todonotes}
\usepackage{enumerate}
\usepackage{enumitem}

\usetikzlibrary{chains}

\renewcommand{\hat}{\widehat}

\mathchardef\mhyphen="2D

\usetikzlibrary{arrows,decorations.pathmorphing,decorations.pathreplacing,positioning,shapes.geometric,shapes.misc,decorations.markings,decorations.fractals,calc,patterns}

\usetikzlibrary{chains}

\tikzset{>=stealth',
	cvertex/.style={circle,draw=black,inner sep=1pt,outer sep=3pt},
	vertex/.style={circle,fill=black,inner sep=1pt,outer sep=3pt},
	star/.style={circle,fill=yellow,inner sep=0.75pt,outer sep=0.75pt},
	tvertex/.style={inner sep=1pt,font=\scriptsize},
	gap/.style={inner sep=0.5pt,fill=white}}

\bibliography{bibliography}
\title{Moduli spaces of framed sheaves on compactified Kleinian singularities}

\author{Søren Gammelgaard}             
\address{Scuola Internazionale Superiore di Studi Avanzati,
	Via Bonomea 265,
	34136 Trieste,
	Italy}
\email{sgammelg@sissa.it}

\begin{document}
		
	\setcounter{tocdepth}{1}

\begin{abstract}
Consider a Kleinian singularity $ \C^2/\Gamma $, where $ \Gamma $ is a finite subgroup of $ \SL_2(\C) $. In this paper, we construct moduli spaces of framed sheaves on a projective Deligne-Mumford stack compactifying the singularity, and we show that these moduli spaces are quasiprojective schemes.
\end{abstract}          
\maketitle                  

\tableofcontents            
		\section{Introduction}\label{sec:ModSpaceIntro}

Let $\Gamma$ be a finite subgroup of $\SL_2(\C)$. Then $\Gamma$ has a tautological action on the affine plane $\C^2 = \Spec \C[x,y]$, and the quotient $\C^2/\Gamma = \Spec \C[x,y]^\Gamma$ is a Kleinian singularity. In previous work (\cite{CGGS}\footnote{Our proof of the main theorem of \cite{CGGS} had a flaw, which was fixed in \cite{CrawYama}.}) we studied the Hilbert schemes $\Hilb^n(\C^2/\Gamma)$, which parametrise $0$-dimensional subschemes of $\C^2/\Gamma$ of length $n$. We can identify such a subscheme with its corresponding ideal of $\C[x,y]$. As a sheaf, such an ideal is coherent, torsion-free, and has rank $1$. It is then a natural question to consider generalising $\Hilb^n(\C^2/\Gamma)$  to cover moduli spaces of coherent, torsion-free sheaves of arbitrary positive rank. One way of doing this is to consider moduli of \emph{framed sheaves}, introduced in \cite{HuybrechtsLehnModuli}. Such a framed sheaf is a pair $(\shf F, \phi_{\shf F})$, where $\shf F$ is a coherent sheaf, and $\phi_{\shf F}$ is a morphism (the \emph{framing morphism}) from $\shf F$ to some fixed \emph{framing sheaf} $\shf G$. 

In \cite{Nakajimabook}, Nakajima considered the moduli spaces of framed sheaves $(\shf F, \phi_{\shf F}$) on $\P^2$ satisfying certain numerical conditions, and showed that they were examples of (Nakajima) quiver varieties. In this case, the framing morphism is an isomorphism $\phi_{\shf F}\colon \shf F|_L\isoto \OO_L^{\oplus r}$, for a line $L\subset \P^2$. This was extended by Varagnolo and Vasserot in \cite{VaVa}, who extended the mentioned action of $\Gamma$ on $\C^2$ to an action on $\P^2$, and considered similar moduli spaces of framed $\Gamma$-\emph{equivariant} sheaves on $\P^2$. In either case, it is easy to see that when the framed sheaves have rank 1, they can be interpreted as ideal sheaves on $\P^2\setminus L = \C^2$. We wish to adapt this strategy to $\C^2/\Gamma$. Paralleling Nakajima and Varagnolo and Vasserot, we will compactify the singularity $\C^2/\Gamma$ to a space $\mathcal X$, and consider the moduli spaces of framed sheaves on $\mathcal X$. The quotient morphism $\C^2\to \C^2/\Gamma$ will extend to a morphism $\pi\colon\P^2\to \mathcal X$.

This compactification $\mathcal{X}$, introduced in \cite{Paper1} (though a similar construction appears in \cite{NakALE}), is however not a variety, but a projective Deligne-Mumford stack. 
It has a distinguished (stacky) divisor $D$, the image of $L$ under $\pi$. We thus wish to construct moduli spaces of framed sheaves $(\shf F, \phi_{\shf F})$ on $\mathcal{X}$, such that $\phi_{\shf F}$ is an isomorphism between $\shf F|_D$ and a fixed coherent sheaf $\shf R$ on $D$.
Providing such a construction is the goal of this paper. Through adapting results of Bruzzo and Sala \cite{BruzzoSala}, we show the following theorem (see \cref{thm:main} for more detail):

\begin{theorem}\label{thm:mainintro}
	Choose a nonnegative integer $n$ and a finite-dimensional $\Gamma$-representation $V$, and let $\shf R$ be the sheaf on $D$ such that $(\pi|_D)^* \shf R =(\OO_L\otimes V)$.

There is a fine moduli space $ \mathbf{Y}_{\mathbf r,n} $ parametrising isomorphism classes of pairs $(\shf F, \phi_{\shf F})$ where $\shf F$ is a coherent, torsion-free sheaf of rank $\dim V$,
 $\phi_{\shf F}$ is an isomorphism $\shf F|_{D}\isoto \shf R$, and $\shf F$ satisfies \mbox{$ \dim H^1(\mathcal X, \shf F\otimes \shf I_{D})= n $}, where $ \shf I_{D} $ is the ideal sheaf of $ D\subset \mathcal X $,  Moreover, $\mathbf{Y}_{\mathbf r,n}$ is a quasiprojective scheme. 
\end{theorem}

This work was originally motivated by a desire to interpret a class of Nakajima quiver varieties, generalising those appearing in \cite{CGGS}. Such quiver varieties do not play a major role in this paper, however, only being necessary for the proof of \cref{lem:boundedHilbPolys}. We do note (\cref{rmk:quiverbij}) that in some cases there is a bijection of the $\C$-points of $Y_{\mathbf r, n}$ with the closed points of a particular quiver variety $\mathfrak{M}_{\theta_0}(\mathbf r_0,n\delta)$ investigated in \cite{Paper1}, but we remain for now unable to determine whether this bijection can be extended to an isomorphism of schemes.

This paper relies extensively on \cite{Paper1} for much of its preliminaries. Like it, this paper heavily adapts, and extends, results of my DPhil Thesis \cite{ID}. Especially, I expand the arguments to consider a larger class of framed sheaves than those considered in \cite{ID}, and in some places I have adapted the proofs to reduce the reliance on quiver varieties.

The first section of this paper contains some of the background information we will need. It does not contain any original material, and as mentioned we refer to \cite{Paper1} for many results. Others are adapted from \cite{BruzzoSala}, and this we write out in some more detail. The second, shorter, section describes the projective stack $\mathcal X$ compactifying $\C^2/\Gamma$, and introduces functors taking framed sheaves on $\mathcal X$ to framed sheaves on $\P^2$ and vice versa. 

In the final section, containing the main mathematical content, we detail the construction of the correct moduli space. The main idea is to use a result of Bruzzo and Sala on the existence of moduli spaces of framed sheaves on projective stacks, furthermore stating that these moduli spaces are quasiprojective schemes, however it is not trivial to match our moduli problem to their result. We handle this by studying the morphism $\pi$, and using its properties to compare framed sheaves on $\mathcal{X}$ with equivariant framed sheaves on $\P^2$. We conclude by showing that the moduli space constructed in \cref{thm:mainintro} indeed generalises $\Hilb^n (\C^2/\Gamma)$.

Ideally, we should be able to find a Theorem stating that the moduli spaces we construct in this paper are Nakajima quiver varieties. Such a result would show, for instance, that the moduli spaces $\mathbf Y_{\mathbf{r},n}$ have symplectic singularities and are irreducible. Furthermore, it would provide an approach to determine the generating series for their Euler characteristics, extending the work already done by Nakajima (\cite{Nak20}) to find the generating series for the Euler characteristics of $\Hilb^n(\C^2/\Gamma)$.

\subsubsection*{Acknowledgments}
The basis of this work is my DPhil thesis, written under the supervision of Balázs Szendr\H{o}i, whom I thank for his encouragement and assistance. I also thank Ugo Bruzzo for reading and commenting on a preliminary version of this paper, and together with Andrea Ricolfi for several conversations on this subject. I was supported by an Aker Scholarship and by a SISSA Mathematical Fellowship.

\subsubsection*{Conventions}
We work throughout over $\C$, the field of complex numbers. Given a group $G$ acting on a scheme (or stack) $X$, $X/G$ denotes the scheme-theoretic quotient if it exists, and $[X/G] $-denotes the stack-theoretic quotient. We write $\Coh{X}$ for the category of coherent sheaves on a scheme (or stack) $X$, and $\ECoh{G}{X}$ for the category of $G$-equivariant sheaves.

		\section{Preliminaries}

\subsection{Coherent sheaves on projective Deligne-Mumford stacks}\label{sec:CohSheavesOnStacks}
We refer to \cite[Sections 2.6-7]{Paper1} (based on \cite{BruzzoSala}, \cite{Nironi}, \cite{Kresch}) for background information and our notation on Deligne-Mumford stacks.

For this section, let $ \mathcal Y $ be a Deligne-Mumford stack, with a coarse moduli space $ f\colon \mathcal Y\to Y $.\todo{space or scheme?}
We use the following definition of a projective Deligne-Mumford stack:
\begin{definition}[{\cite[Corollary 5.4, Definition 5.5]{Kresch}}]
	The stack $ \mathcal Y $ is \emph{projective} if the following two conditions are satisfied:
	\begin{itemize}
		\item $ Y $ is projective,
		\item there is a locally free coherent sheaf $ \shf{V} $ on $ \mathcal{Y} $ such that for any point $x\colon \Spec k\to \mathcal Y $ with $ k $ algebraically closed, the action of the stabiliser group of $ x $ on the fibre $ x^*\shf{V} $ is faithful.
	\end{itemize}
	Such a sheaf $ \shf{V}$ is called a \emph{generating sheaf} for $ \mathcal{Y} $. (See \cite[10]{Kresch} for an explanation of the terminology.)
\end{definition}

\begin{remark}\label{rem:GenSheaf}
	There are other equivalent definitions of what a projective stack is. For instance, we could replace the requirement that there exists a generating sheaf with the requirement that $ \mathcal{Y} $ is a global quotient stack
	(see \cite[Corollary 5.4, Definition 5.5]{Kresch}). 
\end{remark}


For the rest of this section, we additionally assume that $\mathcal Y$ is projective. We also fix an ample sheaf $ \OO_Y(1) $ of $ \OO_Y $-modules, and a generating sheaf $ \shf {V} $ for $ \mathcal{Y} $. (The pair $ (\OO_{Y}(1), \shf{V}) $ is sometimes (e.g. in \cite{BruzzoSala}) called a \emph{polarisation} of the Deligne-Mumford stack $ \mathcal{Y} $.)
\begin{definition}[{\cite[Definition 2.26]{BruzzoSala}}]
	Let $ \shf F $ be a coherent sheaf on a projective stack $ \mathcal{Y} $.  The \emph{modified} Hilbert polynomial of $ \shf F $ is the function $\Z\to \Z $ given by
	\[\nu\mapsto P_{\shf V}(\shf F, \nu) \coloneqq \chi(\mathcal{Y}, \shf F\otimes f^*\OO_Y(\nu)\otimes \shf V\dual) = \chi(Y, f_*(\shf F\otimes\shf V\dual)\otimes \OO_Y(\nu)).\]
	
	We will usually write \emph{Hilbert polynomial} to mean this modified Hilbert polynomial.
\end{definition}
By the last equality (which follows from the projection formula applied to $ f $), we see that  $ P_{\shf V}(\shf F, \nu)  $ is indeed a polynomial in $ \nu $.\todo{Notation on Hilbert polys doesn't match}
Let $ d \coloneqq \dim \Supp \shf F$, which by \cite[Proposition 2.22]{BruzzoSala} equals $ \dim \Supp f_*(\shf F\otimes\shf V\dual) $.
We shall write the Hilbert polynomial as \[  P_{\shf V}(\shf F, \nu)= \frac{\alpha_{\shf{V}, d}\shf F}{d!}\nu^d+\frac{\alpha_{\shf{V}, d-1}\shf F}{(d-1)!}\nu^{d-1}+\cdots + \alpha_{\shf V, 0}.\]

\begin{definition}[{\cite[Definition 4.2]{Nironi}}]\label{def:CastelnuovoMumford} Let $m$ be an integer.
	We say that a coherent sheaf $ \shf F $ on $ \mathcal Y $ is $ m $-regular if $ f_*(\shf F\otimes \shf V\dual) $ is $ m $-regular \ie
	\[  H^i(\mathcal Y, \shf F\otimes \shf V\dual\otimes f^*\OO_Y(m-i))= H^i(Y, f_*(\shf F\otimes \shf V\dual)(m-i))=0 \] for all $ i>0 $.
	The \emph{modified} Castelnuovo-Mumford regularity of $ \shf F $ is the smallest integer $ m $ such that $ \shf F $ is $ m $-regular.
\end{definition}
	We will often write \emph{regularity} to mean modified Castelnuovo-Mumford regularity.

\begin{remark}
	It may seem strange to require the generating sheaf in order to define a Hilbert polynomial and the regularity of $ \shf F $- why not simply consider the Hilbert polynomial $n\mapsto  \chi(\shf F\otimes f^*\OO(n)) $, where $ \OO(1) $ is an ample sheaf on $ Y $? 
	
	The reason is that such a definition would have undesirable properties. For instance, in the case where $ \mathcal{Y} = [\Spec \C/G] $ \ie the classifying space of $ G $, with $ G $ some finite group, the coarse moduli space is simply $ \Spec \C $. In this case, defining the Hilbert polynomial in the naïve way above for a sheaf $ \shf F $ on $ [\Spec \C/G] $, we would get a polynomial only depending on the $ G $-invariants of $ \shf F $.
	
	For this reason, the existence of the generating sheaf -- and especially the properties of the functor $ F_{\shf V}(-)\colon \Coh{\mathcal{Y}}\to\Coh{Y} $ given by $ \shf F\mapsto \shf V\dual\otimes \shf F $ -- play a crucial role in the construction of moduli spaces of sheaves on projective stacks. (See \cite{Nironi} for more on the generating sheaf, and \cite{BruzzoSala}, especially sections 3 and 4 in the latter.)
\end{remark}

\subsection{Properties of coherent sheaves on a projective Deligne-Mumford stack}
Following \cite[Definition 2.18]{BruzzoSala}, we make the following definition:
\begin{definition}\label{def:torfreestack}
	A coherent sheaf $ \shf F $ on a projective stack $ \mathcal Y $ is \emph{torsion-free} if all its nonzero subsheaves are of dimension $ \dim \mathcal Y $.
\end{definition}

We also need a notion of \emph{rank} for coherent sheaves on a projective stack $ \mathcal{Y} $. 
\begin{definition}[{\cite[p.11]{BruzzoSala}}]\label{def:rankOnStack}
	The \emph{rank} $ \rk \shf F $ of a coherent sheaf $ \shf F $ on $ \mathcal{Y} $ is the integer \[ \frac{\alpha_{\shf{V}, d}\shf F}{\alpha_d \OO_Y} = \frac{\alpha_{d}(f_*(\mathcal V\dual \otimes \shf F))}{\alpha_d \OO_Y},\] where $ \frac{\alpha_d \OO_Y}{d!}$ and $ \frac{\alpha_{d}(f_*(\mathcal V\dual \otimes \shf F))}{d!} $ are the leading coefficients of the respective ordinary Hilbert polynomials of $\OO_Y$ and $ f_*(\mathcal V\dual \otimes \shf F) $ on $ Y $.
\end{definition}

We shall, however, be using an alternative formulation easier to work with:
\begin{lemma}\label{lem:rankOnStackLocFree}
	A coherent sheaf $ \shf F $ of $ \OO_{\mathcal{Y}} $-modules has rank $ \rk \shf F = r $ if there is an open subscheme ${U}\injto \mathcal{Y} $ such that $ (\shf F\otimes \shf V\dual)|_{U} $ is locally free of rank $ r $.
\end{lemma}
\begin{proof}
	
	We have $r =\frac{\alpha_{d}(f_*(\shf V\dual \otimes \shf F))}{\alpha_d \OO_Y}$, which by definition is the rank of the sheaf $ f_*(\shf V\dual \otimes \shf F) $ on $ Y $. The restriction of $f$ to $U$ is an isomorphism. It is then a standard computation that on the locus where the sheaf is locally free, it is locally free of rank $ r $.

\end{proof}
Of course, if $ \shf{V} $ is a line bundle, the rank of $ \shf F $ (defined by \cref{def:rankOnStack}) equals the rank of $ \shf F|_{U} $ as a locally free sheaf.

\subsection{Moduli spaces of framed sheaves on projective stacks}\label{sec:ModSpacesProjStacks}
The theory of moduli spaces of framed sheaves on projective Deligne-Mumford stacks has been developed in \cite{BruzzoSala}, we summarise the main points.

We will additionally assume that $ \mathcal Y $ is a normal \emph{irreducible} projective Deligne-Mumford stack, with a coarse moduli \emph{scheme} $ \tau\colon \mathcal Y \to Y $.

\begin{definition}[{\cite[Definitions 3.1, 3.3]{BruzzoSala}}]\label{def:FramedSheaf}
	Let $ \shf F $ be a quasi-coherent sheaf on $ \mathcal Y $. An \emph{$ \shf F $-framed} sheaf on $ \mathcal Y $ is a pair $ (\shf G, \phi_{\shf G}) $, where $ \shf G $ is a coherent sheaf on $ \mathcal Y $, and $ \phi_{\shf G}\colon \shf G\to \shf F $ is a sheaf homomorphism.
	
	A morphism of $ \shf F $-framed sheaves $ (\shf G, \phi_{\shf G})\to (\shf H, \phi_{\shf H}) $ consists of a sheaf morphism $ f\colon \shf G\to\shf H $ such that there is a $c\in \C^\times $ with $ \phi_{\shf G}\circ f=c\phi_{\shf H} $.
\end{definition}
A morphism $  f\colon (\shf G, \phi_{\shf G})\to (\shf H, \phi_{\shf H}) $ of $ \shf F $-framed sheaves is injective, respectively surjective, respectively an isomorphism, if $  f\colon \shf G\to \shf H $ is. 

\begin{definition}\label{def:FramedAlongDivisor}
	Let $ \mathcal D\subset \mathcal Y $ be a smooth integral closed substack of codimension 1, and let $\shf R$ be a locally free sheaf of $\OO_{\mathcal D}$-modules.
A \emph{$(\mathcal D, \shf R)$-framed sheaf} on $ \mathcal Y $ is a pair $ (\shf G, \phi_{\shf G}) $, where $ \shf G $ is a torsion-free sheaf of $\OO_Y$-modules, and $ \phi_{\shf G}\colon \shf G|_{\mathcal D }\isoto\shf R$ is an isomorphism.
\end{definition}
\begin{remark}
	Note that we differ from the definition in \cite[Definition 5.1]{BruzzoSala} on several points:
	\begin{itemize}
		\item We make no requirement on $\dim \mathcal Y $, and we do not require $ \mathcal Y $ to be smooth -- in fact, we will usually work with a singular $ \mathcal Y $;
		\item We do not make any assumptions on the coarse moduli space of $ \mathcal D $;
		\item We do not require $ \shf G $ to be locally free in a neighbourhood of $\mathcal{D}$ -- in the cases we consider, this property follows from the other conditions.\todo{check this one}
	\end{itemize}
\end{remark}
Choose a $ \shf F $-framed coherent sheaf $ (\shf G, \phi_{\shf G}) $ on $ \mathcal Y $, and set \[ \varepsilon(\phi_{\shf G})=\begin{cases}
	0& \textrm{ if }  \phi_{\shf G}=0 ,\\
	1& \textrm{ otherwise}.
\end{cases} \]
Now, let $ d = \dim\Supp \shf G $. Then fix a polynomial \[ \delta(\nu)= \delta_1\frac{\nu^{d-1}}{(d-1)!}+ \delta_2\frac{\nu^{d-2}}{(d-2)!}+\cdots \delta_d\in \Q[\nu] ,\] such that $ \delta_1>0 $.
Then (see \cite[Definition 3.6]{BruzzoSala},) $ \shf G $ is said to be \emph{(semi)stable} (with respect to $ \delta $) if:

\begin{enumerate}
	
	\item $  P_{\shf V}(\shf G', n)(\le)< \frac{\alpha_{\shf V, d}(\shf G')}{\alpha_{\shf V, d}(\shf G)}(P_{\shf V}(\shf G, n)-\varepsilon(\phi_{\shf G})\delta(\nu)) $ for all subsheaves $ \shf G'\subset \ker \phi_{\shf G} $,
	\item $P_{\shf V}(\shf G', n)-\delta(\nu)(\le)< \frac{\alpha_{\shf V, d}(\shf G')}{\alpha_{\shf V, d}(\shf G)}(P_{\shf V}(\shf G, n)-\varepsilon(\phi_{\shf G})\delta(\nu))$ for all subsheaves $ \shf G'\subset \shf G $.
\end{enumerate} Here the notation $ (\le)< $ means: For  \emph{stability}, one uses the equations with $ < $, and for \emph{semistability}, those with $ \le $.

It will turn out to be easier to work with a slightly modified kind of stability for sheaves, namely $ \hat{\mu} $-stability, which is a type of slope stability taking into account the framing sheaf. \todo{rewrite this for paper, maybe delete the original type of framing.}
\begin{definition}[{\cite[Definition 3.9]{BruzzoSala}}]\label{def:muHatStability}
	Let $ (\shf G, \phi_{\shf G}) $ be a $ \shf F $-framed sheaf of dimension $ d $ on $ \mathcal{Y} $, and let $ \delta_1 \in \Q$.
	Then $ \shf G $ is $ \hat{\mu} $-(semi)stable with respect to $ \delta_1 $ if the following conditions hold:
	
	\begin{enumerate}
		\item $ \ker \phi_{\shf G} $ is torsion-free,
		\item \[ \frac{\alpha_{\shf{V}, d-1}(\shf G')}{\alpha_{\shf{V}, d}(\shf G')} (\le)< \frac{\alpha_{\shf{V}, d-1}(\shf G)-\varepsilon(\phi_{\shf G})\delta_1}{\alpha_{\shf{V}, d}(\shf G)} \] for all $ \shf G'\subset \ker \phi_{\shf G} $, and
		\item \[ \frac{\alpha_{\shf{V}, d-1}(\shf G')-\delta_1}{\alpha_{\shf{V}, d}(\shf G')} (\le)< \frac{\alpha_{\shf{V}, d-1}(\shf G)\delta_1}{\alpha_{\shf{V}, d}(\shf G)} \] for all $ \shf G'\subset \shf G $ with $ \alpha_{\shf G', d}<\alpha_{\shf G, d} $,
	\end{enumerate}with the same convention as before regarding the notation $ (\le)< $.
	
	The number $ \hat{\mu}(\shf G, \phi_{\shf G})\coloneqq   \frac{\alpha_{\shf{V}, d}(\shf G)-\delta_1}{\alpha_{\shf{V}, d-1}(\shf G)}$ is called the \emph{framed hat-slope} of $ (\shf G, \phi_{\shf G}) $.
\end{definition}

Let $ \delta(\nu) \in \Q[\nu]$ be a polynomial such that $ \delta_1/(d-1)! $ is the degree $ (d-1) $-coefficient of $ \delta $. We then have~\cite[10]{BruzzoSala} a chain of implications for a sheaf $ \shf F $:
\begin{equation} \hat\mu \textrm{-stable}\implies \textrm{stable}\implies\textrm{semistable}\implies \hat{\mu}\textrm{-semistable},  \end{equation} 
where $ \hat{\mu} $-(semi)stability is considered with respect to $ \delta_1 $, (semi)stability with respect to $ \delta $.

\begin{remark}\label{rem:mustableIsStable}
	Especially, given a flat family of stable sheaves on $ \mathcal Y $, being $ \hat\mu $-stable is an \emph{open condition}, by the same argument as in \cite[Proposition 3.A.1]{HuybrechtsLehn}. (This means: If $ \shf G $ is a sheaf on $ \mathcal Y\times B $, flat over $ B $, such that for every $ b\in B $, $ \shf G_b $ is stable, then the set $ \{b\in B\mid \shf G_b\textrm{ is $ \hat{\mu} $-stable} \} $ is open.)
\end{remark}

Now, let \[ \mathcal{M}_{\mathcal{Y}}(\shf V, \OO_Y(1), P_0,\shf F, \delta(\nu)) \] be the functor associating to any scheme $ S $ of finite type the set of isomorphism classes of flat families of $ \delta(\nu) $-stable $ \shf F $-framed sheaves with Hilbert polynomial $ P_0 $ on $ \mathcal Y $, parametrised by $ S $.

We then have
\begin{theorem}[{\cite[Theorem 4.15]{BruzzoSala}}]\label{thm:ModSpaceExists}
	
	The functor $ \mathcal{M}_{\mathcal{Y}}(\shf V_\mathcal{Y}, \OO_Y(1), P_0,\shf F, \delta(\nu)) $ is represented by a quasiprojective scheme ${\mathbf M}_{\mathcal{Y}}(\shf V_\mathcal{Y}, \OO_Y(1), P_0,\shf F, \delta(\nu))$.
\end{theorem}

		\section{Compactified Kleinian singularities and sets of framed sheaves}

Choose, once and for all, a finite group $\Gamma\subset \SL_2(\C)$. It acts tautologically on $\C^2\coloneqq \Spec \C[x,y]$. The scheme $\C^2/\Gamma\coloneqq \Spec \C[x,y]^\Gamma$ is a \emph{Kleinian singularity}. Let $p\in \C^2/\Gamma$ be the unique singular point, and denote its preimage in $\C^2$ by $o$. To make a compactification of $\C^2/\Gamma$, we embed it into a projective stack in the following way: Extend the $\Gamma$-action to $\P^2\coloneqq \Proj \C[x,y,z]$ by setting $g\cdot z = z$ for all $g\in \Gamma$. We consider $\C^2$ as a subscheme of $\P^2$ by the obvious embedding. We also write $L\coloneqq\P^2\setminus \C^2= (z=0)$ , the 'line at infinity'. Let $U_L\coloneqq \P^2\setminus \{o\}$.
\begin{definition}\label{def:StackX}
	 We set $\mathcal{X}$ to be the fibre product
\[\mathcal{X}\coloneqq \C^2/\Gamma \times_{(\C^2\setminus \{o\})/\Gamma}[U_L/\Gamma].\] 
\end{definition}
For more details on this construction, see \cite[Definition 3.1]{Paper1}. 
Intuituvely, we have taken the `scheme-theoretic' quotient around $o$, and the `stack-theoretic' quotient around $L$. 

There is an induced morphism $\pi\colon \P^2\to \mathcal X$, factoring through the quotient stack $[\P^2/\Gamma]$.

From now on, we focus on the stack $\mathcal X$.
\begin{proposition}[{\cite[Lemma 3.2, Lemma 3.4, Proposition 3.5]{Paper1}}]\label{prop:XisProjective}
    $\mathcal{X}$ is a projective Deligne-Mumford stack of finite type, containing the singular scheme $\C^2/\Gamma $ as an open substack. It has a unique singular point, namely $p$. Its coarse moduli scheme is $\P^2/\Gamma$. The sheaf $\OO_{\mathcal X}(D)$ is a generating sheaf for $\mathcal{X}$.
\end{proposition}

Set $D\coloneqq \mathcal{X}\setminus \C^2/\Gamma$.
Various morphisms are induced by the definition of $\mathcal{X}$, which we notate in this way:
\begin{equation}\label{diag:bigDiagram}
	\begin{tikzcd}
		L \arrow[dd, two heads, "\pi|_{L}"'] \arrow[r, hook, "i"]                 & \P^2 \arrow[d, two heads, "k"] \arrow[dd, two heads, bend left=49, "\pi"] \\
		& {[\P^2/\Gamma]} \arrow[d, two heads, "q"]               \\
		D\arrow[r, hook, "j"] \arrow[ru, hook] & \mathcal X \arrow[d, two heads, "c"]                             \\
    & \P^2/\Gamma                                       
	\end{tikzcd}.
\end{equation}

Let $Q_0\coloneqq \{\rho_0,\dots, \rho_k\}$ be the irreducible representations of $\Gamma$, with $\rho_0$ as the trivial representation. Given any $\rho_m$, the $\Gamma$-equivariant sheaf $\OO_L\otimes \rho_m$ descends along $\pi|_L$ to a sheaf on $D$, we denote this sheaf by $\shf R_m$.
\begin{notation}
\emph{For the rest of the paper}, we fix a vector $\mathbf r = (r_0,\dots,r_k)\in \Z_{\ge 0}^{Q_0}$, and a locally free sheaf $\shf R\coloneqq \bigoplus_{i=0}^k \shf R_i^{\oplus r_i}$ of $\OO_D$-modules, and we set $r=\sum r_i$.
\end{notation}

In \cite{Paper1}, we introduced sets $X_{r,\mathbf v}$ and $Y_{r,n}$ of isomorphism classes of framed sheaves on $\P^2$ and $\mathcal X$, respectively. We now modify those:

\begin{definition}\label{def:Xrv}\label{def:Yrn}
	Choose a vector $\mathbf v= (v_0,\dots, v_k)\in \Z_{\ge 0}^{Q_0}$, and set $n = v_0$.
	
	Let $ X_{\mathbf r,\mathbf{v}} $ be the set of isomorphism classes of $ (L,\bigoplus_{i=1}^k \shf (\OO_L\otimes \rho_i)^{\mathbf r_i}) $-framed equivariant sheaves of $ \OO_{\P^2} $-modules $ (\shf F, \phi_{\shf F}) $, where $ \shf F $ 
	\begin{itemize}
		\item is coherent,
		\item is torsion-free of rank $ r\coloneqq \sum r_i $,
		\item satisfies $H^1(\P^2, \shf F\otimes \shf I_{L})\isoto \sum \rho_i^{v_i} $, where $ \shf I_{L} $ is the ideal sheaf of $ L\subset \P^2 $.
	\end{itemize} 

	Let $ Y_{\mathbf r,n} $ be the set of isomorphism classes of $ (D, \bigoplus_{i=1}^k \shf R_i^{r_i}) $-framed $ \OO_{\mathcal X} $-modules $ (\shf F, \phi_{\shf F}) $, where $ \shf F $ 
	\begin{itemize}
		\item is coherent,
		\item is torsion-free of rank $ r\coloneqq \sum r_i $,
		\item satisfies $ \dim H^1(\mathcal X, \shf F\otimes \shf I_{D})= n $, where $ \shf I_{D} $ is the ideal sheaf of $ D\subset \mathcal X $.
	\end{itemize}  
\end{definition}

\begin{remarks}\label{rmk:SchStructX}
	The set $ X_{\mathbf r,\mathbf{v}}$ is the set of $\C$-valued points of a natural moduli space $\mathbf{X}_{\mathbf{r,v}}$ which is a quasiprojective scheme. To see this, by \eg \cite[Theorem 0.1]{HuybrechtsLehnModuli} the moduli space of $(L,\OO_L^r)$-framed sheaves $\shf G$ of $\OO_{\P^2}$-modules satisfying $ \dim H^1(\P^2, \shf F\otimes \shf I_{L})= \sum {v_i} $ is a quasiprojective scheme, let us denote it $X_{r, v}$. 
	Then $ \mathbf X_{\mathbf r,\mathbf{v}}$ is a closed subscheme of $X_{r, v}$. 
	
	The set $X_{\mathbf r,\mathbf{v}}$ was introduced in \cite{VaVa}, but it was not there given an intrinsic scheme structure.
	\end{remarks}
	
\begin{remark}\label{rmk:locFreeNbhd}
	A framed sheaf lying in either $X_{\mathbf r,\mathbf{v}} $ or $ Y_{\mathbf r,n}$ is locally free in some neighbourhood of respectively $L$ or $D$, as we showed in \cite[Lemmata 4.4, 4.11]{Paper1}.
\end{remark}

We will often leave the framing isomorphisms $ \phi_{\shf E}, \phi_{\shf F} $ implicit, and write ${\shf E}\in X_{\mathbf r, \mathbf v},\ \shf F \in Y_{\mathbf r,n} $, simply calling each a `framed sheaf'. 

There are two important functors we will be using:
\begin{definition}\label{def:UpDownFunctors}
	\item Let $\pi^T\colon \Coh{\mathcal{X}}\to \ECoh{\Gamma}{\P^2}$ be the functor taking a sheaf $\shf F$ on $\mathcal{X}$ to the quotient of $\pi^*\shf F$ by its torsion subsheaf.
	
	\item Let $D\colon \ECoh{\Gamma}{\P^2}\to \Coh{\mathcal{X}}$ be the functor taking a $\Gamma$-equivariant sheaf $\shf E$ on $\P^2$ to the sheaf on $\mathcal X$ obtained by first descending along $k$, then taking the direct image along $q$.
\end{definition}

The two functors just introduced have nice properties when extended to framed sheaves:
\begin{lemma}
Let $\shf F\in Y_{\mathbf r, n}$, and let $\shf E\in X_{\mathbf{r}, \mathbf{v}}$. Then the framing morphism of $\shf F$ induces a framing morphism of $\pi^T \shf F$, and the framing morphism of $\shf E$ induces a framing morphism of $D(\shf E)$. 
We can thus extend the two functors to functors between categories of framed sheaves, keeping the same notation. Indeed, we have $D(\shf E)\in Y_{\mathbf r, n} $ for $n ={v}_0$, and there is a vector $\mathbf u\in \Z_{\ge 0}^{Q_0}$ such that $u_0 = n$, and $\pi^T(\shf F)\in X_{\mathbf r,\mathbf u}$. Furthermore, $D$ is exact.

\end{lemma}
\begin{proof}
These statements are straightforward extensions of Proposition 4.7 and Corollary 4.14 from \cite{Paper1}, and the final statement is shown following Remark 4.2], \emph{ibid}.
\end{proof}

\begin{lemma}\label{lem:upanddowniso}
$D\circ \pi^T(\shf F) = D\circ \pi^*(\shf F)=\shf F$ for any $\shf F\in Y_{\mathbf{r}, n}$. For any $\shf E\in X_{\mathbf r, \mathbf v}$, there is a canonical injective morphism $\pi^T(D(\shf E))\injto \shf E$.
\end{lemma}
\begin{proof}
    Verbatim the same as in \cite[Lemmata 4.12, 4.13]{Paper1}.
\end{proof}

\section{A scheme structure for $ Y_{\mathbf r,n} $}\label{sec:YrnIsModScheme}\label{chap:ConstructModuliSpace}

Our goal is to prove \cref{thm:mainintro}, which we restate here in more detail:
\begin{theorem}\label{thm:main}
	There is a quasiprojective scheme  $ \mathbf{Y}_{\mathbf r,n}, $ which is a fine moduli space for the functor given by assigning to a scheme $S$ of finite type the set
		\[ \mathbf{Y}_{\mathbf r,n}(S) =\left\{{\begin{split}
			&\textrm{isomorphism classes of flat families of framed sheaves $(\shf F, \phi_{\shf F})\in \mathbf{Y}_{\mathbf{r}, n}$}\\
			&\textrm{parametrised by $S$.} 
	\end{split}}\right\}.\] 
	\end{theorem}

Recall that the vector $\mathbf r$ uniquely determines $\shf R$. Any finite-dimensional $\Gamma$-representation $V$ decomposes into a direct sum of irreducible representations, and is thus also determined by $\mathbf{r}$ -- so the statement of \cref{thm:main} is indeed a version of \cref{thm:mainintro}.
The $\C$-valued points of $\mathbf{Y}_{\mathbf r,n}$ can be identified with $Y_{\mathbf r, n}$.

The proof of \cref{thm:main} will take up most of the remaining paper.
Our strategy is to use \cref{thm:ModSpaceExists}, and show that  $\mathbf{Y}_{\mathbf r,n}$ is an open subscheme of a moduli space of the type ${\mathbf M}_{\mathcal{Y}}(\shf V_\mathcal{Y}, \OO_Y(1), P_0,\OO_{D}^r , \delta(\nu))$. Thus we must show that:
\begin{itemize}
	\item there is a $ \delta_1\in \Q $ such that any $ \shf F\in Y_{\mathbf r,n} $ is $ \hat{\mu}$-stable with respect to $ \delta_1 $.
	\item the modified Hilbert polynomial $ P_{\OO_{X}(D)}(\shf F,n) $ is independent of the choice of  $ \shf F\in Y_{\mathbf r,n} $, and
	\item there is a bound for the modified Castelnuovo-Mumford regularities of the framed sheaves in $ Y_{\mathbf r,n} $.
	
\end{itemize} 

We start by determining the Hilbert polynomials.
\subsection{Hilbert polynomials of framed sheaves on $ \mathcal X $}\label{sec:HilbPolys}\label{subsec:EqSheavesOnP2}
The isomorphism classes of $ \Gamma $-equivariant line bundles on $ \P^2 $ form a group under $ \otimes $, which we denote $ \Pic_\Gamma(\P^2) $. This group is usually called the \emph{equivariant Picard group}, see e.g. \cite{Sadhu} or \cite{MFK}.

Forgetting the equivariant structure gives a homomorphism $ t\colon\Pic_\Gamma(\P^2)\to \Pic(\P^2) $.

For instance, the line bundle $ \OO_{\P^2}(L) $ has a $ \Gamma $-equivariant structure induced by the action of $ \Gamma $ on $ \P^2 $, and it maps to $ \OO_{\P^2}(1) $ under $ t $. Especially, the ideal sheaf sequence \[0\to \OO_{\P^2}(-L)\to \OO_{\P^2}\to \OO_{L}\to 0 \] is equivariant.
By \cite[Proposition 4.2]{KnopKraftVust}, the set of elements of $ \Pic_\Gamma(\P^2)$ that are inverse images under $\pi\circ c$ of line bundles on $\P^2/\Gamma$ consists exactly of those line bundles $ \shf L $ such that, for each closed point $x\in \P^2 $, the action of the stabiliser group $ \Gamma_x $ on fibres of $ \shf L $ is trivial.

For any point $ x\in\P^2 $, the action of the stabiliser group (for the action of $\Gamma$ on $\P^2$) on the fibre of $ \OO_{\P^2}(L) $ at $x$ factors through the action of some cyclic group, because the fibres are one-dimensional. Therefore there must be some minimal $ m $ such that the tensor power $\OO_{\P^2}(mL) =\OO_{\P^2}(L)^{\otimes m}$ lies in the inverse image of $ \Pic(\P^2/\Gamma) $ under $\pi\circ c$. This tensor power must furthermore be (very) ample. 

From now on, fix $\OO_{\P^2/\Gamma}(1)$ to be a line bundle on $\P^2/\Gamma$ such that \[ \pi^*c^*\OO_{\P^2/\Gamma}(1) = \OO_{\P^2}(L)^{\otimes m} .\] Because $ \P^2/\Gamma $ is projective and $ c\circ\pi\colon \P^2\to \P^2/\Gamma $ is finite, by \cite[Exercise III.5.7.d]{Har77}, $ \OO_{\P^2/\Gamma}(1)$ is ample.

\subsubsection*{Equivariant Euler characteristics}\label{sec:equivEuler}
Following Ellingsrud and Lønsted~\cite{EllLon}, we introduce an `equivariant Euler characteristic'.\todo{rewrite with Lefschetz Trace}

Recall that the standard Euler characteristic $ \chi $ for coherent sheaves on a scheme $ X $ can be interpreted as a homomorphism $ \chi\colon K(X)^f\to \Z $, where $ K(X)^f $ is the subgroup of the Grothendieck group $ K(X) $ given by the equivalence classes of sheaves $ \shf F $ with finite-dimensional cohomology groups $ H^i(X, \shf F) $.
Now, let $ G $ be a finite group acting on $ X $, let $ \operatorname{Irr}_G $ be its set of irreducible representations, and let $ R G $ be the representation ring of $ G $ over $ \C $. If $ \shf F $ is a $ G $-equivariant sheaf on $ X $, each cohomology group $ H^i(X, \shf F) $ carries an induced $ G $-representation structure, and we denote its image in $ R G $ by $ [H^i(X, \shf F)] $. 

Let $ K_G(X)^f $ be the finite-dimensional equivariant $ K $-theory of $ X $: Take the free abelian group $F$ on the isomorphism classes of $ G $-equivariant coherent $ \OO_{\mathcal{X}} $-modules with finite-dimensional cohomology groups.  For every short exact sequence of $G$-equivariant homomorphisms \[ 0\to \shf F\to \shf E\to \shf G\to 0\] of such equivariant sheaves, consider the expression $ [\shf F]+[\shf G]-[\shf E] $ in $F$. Then $ K_G(X)^f $ is the quotient of $F$ by the subgroup generated by the set of all such isomorphism classes.
There is then a map $ \chi_G\colon K_G(X)\to RG $, defined by 
\[ \shf F\mapsto \sum_{i\ge 0} (-1)^i [H^i(X, \shf F)].\] (This is sometimes called a \emph{Lefschetz trace}, see e.g. \cite[Definition 1.1]{EllLon}.)

It is then clear that the ordinary Euler characteristic $ \chi $ is the composition of $ \chi_G $ with the homomorphism $ R G\to \Z $ defined by $ \sum_{\rho_j\in \operatorname{Irr}_G} n_j[\rho_j]\mapsto \sum_{\rho_j\in \operatorname{Irr}_G}n_j\dim \rho_j $ -- here we have used Maschke's theorem to say that any $ G $-representation splits into a direct sum of irreducible representations.
Furthermore, there are `isotypic projection' group homomorphisms $ p_j\colon RG\to \Z $, given by $ p_j\colon \sum_{\rho_j\in \operatorname{Irr}_G} n_j[\rho_j]\mapsto n_j$.

We now set $ G=\Gamma, X = \P^2 $, and write $ \chi_0 \coloneqq p_{0}\circ\chi_\Gamma$.
Let's use this to compute Euler characteristics for sheaves on our stack $ \mathcal{X} $.
\begin{lemma}\label{lem:computingEulerChars}
	Let $ \shf E\in \Coh{\mathcal{X}} $. Let $ \shf F $ be any $ \Gamma $-equivariant coherent sheaf on $ \P^2 $ such that $ D(\shf F) =\shf E$. Then
	\[ \chi(\mathcal{X}, \shf E) = \chi_0(\P^2, \shf F) .\]
\end{lemma}

\begin{proof}
	This follows from the fact that $ \mathcal{X} $ and $ [\P^2/\Gamma] $ share $ \P^2/\Gamma $ as their coarse moduli space, together with  the fact that $H^i([\P^2/\Gamma], \shf G) = H^i(\P^2, k^*\shf G)^\Gamma$ for any sheaf on $[\P^2/\Gamma]$ (see \eg \cite[Lemma 2.19]{Paper1}).
\end{proof}

\begin{proposition}\label{prop:sameEquivHilbFunct}
	Let $ \shf F\in Y_{\mathbf r,n} $.
	
	The function 
	\[ H_\Gamma(\shf F,\nu)\coloneqq \nu\mapsto \chi_0(\P^2, \pi^T\shf F(\nu L)) \] does not depend on the choice of $ \shf F $.
\end{proposition}

\begin{proof}
	
	The sheaf $ \pi^T\shf F $ lies in some $ X_{\mathbf r, \mathbf{v}} $ such that $ \mathbf v_0=n $\todo{this one might not survive arbitrary cornering} by \cite[Corollary 4.14]{Paper1}, and so
	 $ \dim(H^1(\P^2, \pi^T\shf F(-L)))^\Gamma=n $. 
	
	It follows from \cite[Lemma 2.4]{Nakajimabook} that 
	\[ H^i(\P^2, \pi^T\shf F(-p L))=0 \] for $ i=0,2; p=1,2 $. 	Especially, we find 
	\[ \chi_0(\pi^T\shf F(-L))=n.\]
	
	We can for every $ \nu\in \Z $ set up a short exact $ \Gamma $-equivariant sequence
	\[ 0\to \pi^T\shf F((\nu -1)L)\to \pi^T\shf F(\nu L)\to \pi^T\shf F|_{L}(\nu L)\to 0 ,\] which is exact because $ \pi^T\shf F $ by \cref{rmk:locFreeNbhd} is locally free in a neighbourhood of $ L $.
	Now $ \pi^T\shf F|_{L}= \bigoplus_{\rho_i\in Q_0} (\OO_L\otimes \rho_i)^{r_i}$,
	which shows 
	that 
	\[ \chi_0(\P^2, \pi^T\shf F(\nu L))  = \chi_0(\P^2, \pi^T\shf F((\nu-1) L))+  \chi_0\left(\bigoplus_{\rho_i\in Q_0} (\OO_L\otimes \rho_i)^{r_i}(\nu L))\right).\]
	Since $\mathbf r = (r_i)$ is fixed, this implies that for any $ \nu $, $ \chi_0(\P^2, \pi^T\shf F(\nu L)) $ can be iteratively computed from the function $\nu\mapsto \left(\bigoplus_{\rho_i\in Q_0} (\OO_L\otimes \rho_i)^{r_i}(\nu L))\right)$ and $ \chi_0(\P^2, \pi^T\shf F(- L))=n$, both of which are independent of the choice of $\shf F$. This concludes the proof.

\end{proof}
The function $ H_\Gamma(\shf F,\nu) $ is not in general a polynomial in $\nu$. A slight adaptation of the next proof shows that it is however a \emph{quasi-polynomial} \ie there is a set of numerical polynomials $ p_1,\dots, p_m \in \Q[\nu]$ such that $ H_\Gamma(\shf F,\nu) = p_i(\nu)$ for $ \nu \equiv i \operatorname{mod} m $. We will not make use of this.
\begin{corollary}\label{cor:sameHilbPoly}
	All framed sheaves $ \shf F\in Y_{\mathbf r,n} $ have the same modified Hilbert polynomial.
\end{corollary}

\begin{proof}
	Let $ \shf F\in Y_{\mathbf r,n} $.
	
	Since the dual of $ \OO({D}) $ is $ \shf I_{D} $, we have
	\[P_{\OO(D)}(\shf F, \nu)= \chi(\mathcal{X}, \shf F\otimes \shf I_{D}\otimes c^*\OO_{\P^2/\Gamma}(\nu)).\]
		To compute these Euler characteristics, we shall work on $ \P^2 $.
	First, we make the following \emph{Claim:}
	\begin{multline} D(\pi^*\big(\shf F\otimes \shf I_{D}\otimes c^*\OO_{\P^2/\Gamma}(\nu)\big))=D(\pi^T\big(\shf F\otimes \shf I_{D}\otimes c^*\OO_{\P^2/\Gamma}(\nu)\big))\\=\shf F\otimes \shf I_{D}\otimes c^*\OO_{\P^2/\Gamma}(\nu).\end{multline}
	To see this, it is enough to restrict to some open substack $ U \subset \mathcal{X} $ trivialising \newline $\shf I_{(D)}\otimes c^*\OO_{\P^2/\Gamma}(\nu) $, such that $ p\in U $. Here we have 
	\[ D\left(\pi^*(\shf F\otimes \shf I_{D}\otimes c^*\OO_{\P^2/\Gamma}(\nu))\right)|_U \isoto D(\pi^*(\shf F|_{U})) ,\]
	\[ D\left(\pi^T(\shf F\otimes \shf I_{D}\otimes c^*\OO_{\P^2/\Gamma}(\nu))\right)|_U \isoto D(\pi^T(\shf F|_{U})) , \] and the objects on the right-hand side of the two above isomorphisms both equal $ \shf F|_U $  by \cref{lem:upanddowniso}. Take another open $\Gamma$-invariant substack $V\subset \mathcal X\setminus \{p\}$ with $V\cup U = \mathcal X$ and $D\subset V$, then  $D\circ \pi^T|_V = D\circ\pi^*|_V = \id$. The transition maps on $U\cap V$ are unchanged by these last two equalities, and we have shown the Claim.
	
	By construction of $ \OO_{\P^2/\Gamma}(1) $, we have 
	\begin{multline} \pi^T\left(\shf F\otimes \shf I_{D}\otimes c^*\OO_{\P^2/\Gamma}(\nu)\right)=\pi^T\shf F\otimes \shf I_{L}\otimes \OO_{\P^2}(\nu m L) \\
		=\pi^T\shf F(\nu(m-1)L).\end{multline}
	Now we can apply \cref{lem:computingEulerChars}, which tells us that 
	\[\chi(\mathcal{X}, \shf F\otimes \shf I_{(D)}\otimes c^*\OO_{\P^2/\Gamma}(\nu))=\chi_0(\P^2,\pi^T\shf F(\nu(m-1)L)).
	\]
	So we find that $ P_{\OO(D)}(\shf F, \nu)$ is given by \[ \nu\mapsto H_\Gamma(\pi^T\shf F,\nu m-1),\] which by \cref{prop:sameEquivHilbFunct} is the same for all $ \shf F\in Y_{\mathbf r,n} $.
	
\end{proof}

\subsection{Boundedness and regularity}\label{sec:boundedYrn}
We will now show that the set $ \{\pi^T\shf F|\shf F\in Y_{\mathbf r,n}\} $ is \emph{bounded}\ie the set of Hilbert polynomials \[ \{P(\pi^T\shf F, \nu)\mid\shf F\in Y_{\mathbf r,n}\} \] is finite. 
For a coherent sheaf $ \shf G $ on $ \P^2 $, set $  P(\shf G, \nu)\coloneqq \chi(\P^2, \shf G\otimes \OO_{\P^2}(\nu)) $ to be the (ordinary) Hilbert polynomial.
\begin{remark}\label{rem:bounded}
	The more usual definition, in \cite[Definition 4.10]{Nironi} (for Deligne-Mumford stacks) originally for schemes in \cite[XIII-1.12]{SGA6}, would say: $ Y_{\mathbf r,n} $ is \emph{bounded} if there is a scheme $ T $ of finite type over $ \C $ and a coherent sheaf $ \shf G $ on $ \mathcal X\times T $ such that every element of $ Y_{\mathbf r,n} $ is contained in the fibres
	of $ \shf G $ over closed points of $ T $. The equivalence with our definition is proved in \cite[Theorem 4.12]{Nironi}, assuming that the set of sheaves $ Y_{\mathbf r,n} $ satisfies a cohomological condition.
\end{remark}

\begin{lemma}\label{lem:boundedHilbPolys}
	The set of Hilbert polynomials $\{P(\pi^T\shf F, \nu) \}$ for $ \shf F\in Y_{\mathbf r,n} $ is finite.
\end{lemma}
\begin{proof}
	This is because the Hilbert polynomial is completely determined by the rank of $ \pi^T\shf F $ and the Euler characteristics $ \chi(\P^2,\pi^T\shf F(-L))$ and $ \chi(\P^2,\pi^T\shf F(-2L)) $. 
	Let $ \mathbf{v} =H^1(\P^2, \pi^T(\shf F)(-L))$. The Riemann-Roch theorem for surfaces shows that $H^1(\P^2, \pi^T(\shf F)(-L)) = H^1(\P^2, \pi^T(\shf F)(-2L))$. Now $\pi^T\shf F\in X_{\mathbf r, \mathbf{v}}$, which can be identified with the closed points of a specific \emph{Nakajima quiver variety} $\mathfrak{M}_{\theta}(\mathbf v,\mathbf r)$ by \cite[Theorem 1]{VaVa}. By \cite[Proposition 3.6]{BertschGyengeSzendroi}, there is a finite number of vectors $\mathbf v$ satisfying $\mathbf v_0 =n$ such that $\mathfrak{M}_{\theta}(\mathbf v,\mathbf r)$ is nonempty. Thus there are finitely many possible values for $\mathbf v$, so there is a finite number of possible Hilbert polynomials.
\end{proof}
This proof is the only place in the present paper where we explicitly use quiver varieties.\footnote{The lemma raises an unrelated, but interesting problem: Suppose there is a finite group $G$ acting on a scheme $Y$. Let $\shf F$ be a $G$-equivariant sheaf of $\OO_Y$-modules. In general, it is impossible to determine or estimate $H^i(Y, \shf F)$ given $H^i(Y,\shf F)^G$. What conditions must $\shf F$ satisfy before such a prediction becomes possible? The lemma shows that at least torsion-freeness, together with framing and coherence will work.}

We move to the (modified) regularity of the sheaves in $ Y_{\mathbf r,n} $.

\begin{lemma}\label{lem:boundedregularity}
	There is an integer $ p $, depending only on $ r $ and $ n $, such that every framed sheaf $ \shf F\in Y_{\mathbf {r},n} $ is $ p $-regular.
\end{lemma}
By \cref{def:CastelnuovoMumford}, we need to show that there is an $ p\in \Z $ such that for any $ \shf F\in Y_{\mathbf{r},n} $, the groups

\[ H^i(\mathcal X, \shf F\otimes \shf I_{D}\otimes c^*\OO_{\P^2/\Gamma}(p-i)) \]  vanish for all $ i $.

For any numerical polynomial $ P\in \Q[\nu] $, set $Y_P\subset Y_{\mathbf r,n}$ to be the subset given by
\[Y_P =\{\pi^T\shf F\mid \shf F\in Y_{\mathbf r,n}, P(\pi^T\shf F, \nu)=P\} .\]

Then, by \cite[Theorem 2.5]{BruMar}, each $ Y_P $ is bounded. This means that given a polynomial $ P $, each $ \pi^T\shf F\in Y_P $ is $ p_P $-regular for some $ p_P $ depending on $ P $. But since the set of nonempty $ Y_P $ is finite by \cref{lem:boundedHilbPolys},  $ \pi^T\shf F $ is $ p' \coloneqq \max\{p_P\} $-regular for each $ \shf F\in Y_{\mathbf{r},n} $.
From the proof of \cref{cor:sameHilbPoly}, we know that the cohomology groups 
\[ H^i(\mathcal X, \shf F\otimes \shf I_{D}\otimes c^*\OO_{\P^2/\Gamma}(m'-i))\] 
vanish when the groups 
\begin{equation}\label{eq:cohomgroups} H^i(\P^2, \pi^T(\shf F\otimes\shf I_{D}\otimes c^*\OO_{\P^2/\Gamma}(m'-i))) =H^i(\P^2, \pi^T\shf F\otimes \OO_{\P^2}\big((m(m'-i)-1)\big))
\end{equation} do. 

We set $ p$ to be an integer greater than $ p'/m+2 $.
Then, for $ i=1 $ \eqref{eq:cohomgroups} becomes 
\[ H^1(\P^2, \pi^T\shf F\otimes \OO_{\P^2}((mm'-m-1)).\] If we choose $ m'>p $, then $ mm'-m>p' $, so by $ p' $-regularity of $ \pi^T\shf F $, this group vanishes.
A similar computation shows that if $ m'>p $, \[ H^2(\P^2, \pi^T\shf F\otimes \OO_{\P^2}((mm'-2m-2))=0 .\]
It follows that all $ \shf F\in Y_{\mathbf{r},n} $ are $ p $-regular.

\subsection{Stability of framed sheaves on $ \mathcal X $}

We need one more ingredient before we can construct the moduli space: we need the sheaves in $ Y_{\mathbf r,n} $ to fulfill a $ \hat\mu $-stability condition.
\begin{lemma}\label{lem:sheavesInYrnAreStable}
	There is a $ \delta_1\in \Q $ such that each $ \shf F\in Y_{\mathbf r,n} $ is $ \hat{\mu} $-stable with respect to $ \delta_1 $.
\end{lemma}

\begin{proof}
	Let $ \shf F\in Y_{\mathbf r,n} $.
	
	Because we know that the Hilbert polynomial $ P_{\OO_{X}(D)}(\shf F, \nu) $ is independent of $ \shf F $, and they are all $ p $-regular,
	we can apply the \emph{Grothendieck lemma} in its stacky version \cite[Lemma 4.13, Remark 4.14]{Nironi}. This tells us that there is a constant $ C $, depending only on the modified Hilbert polynomial and modified Castelnuovo-Mumford regularity of $ \shf F $, such that any purely 2-dimensional subsheaf $\shf F'\subset \shf F $ satisfies $ \hat{\mu}(\shf F')\le C $. Fix such a $ C $.
	
	Let then $ \shf F'\subset \shf F $. Because $ \shf F $ is torsion-free, $ \shf F' $ is pure of dimension 2. Now, the modified Hilbert polynomial and the modified Castelnuovo-Mumford regularity of $ \shf F $ do not depend on the choice of $ \shf F $, and thus neither does its framed hat-slope $ \hat{\mu}(\shf F) $. We will thus write $ \hat{\mu}\coloneqq\hat{\mu}(\shf F) $.
	
	We can rewrite condition 2 in \cref{def:muHatStability} to
	\begin{equation}\label{eqn:muHatStableOnX}
		\hat{\mu}(\shf F')-\hat\mu(\shf F) \le \delta_1\left(\frac{1}{\alpha_{\OO_{\mathcal X}(D), 2}(\shf F')}-\frac{1}{\alpha_{\OO_{\mathcal X}(D), 2}(\shf F)}\right).\end{equation}
	
	 We then have $ \hat\mu{(\shf F')}\le C $, which implies that \cref{eqn:muHatStableOnX} will hold whenever \[ \delta_1\ge \frac{(C-\hat{\mu})(\alpha_{\OO_{\mathcal X}(D), 2}(\shf F')\alpha_{\OO_{\mathcal X}(D), 2}(\shf F))}{\alpha_{\OO_{\mathcal X}(D), 2}(\shf F)-\alpha_{\OO_{\mathcal X}(D), 2}(\shf F')}\]
	
	Note that the maximal value of the right-hand side is attained when $ \alpha_{\OO_{\mathcal X}(D), 2}(\shf F') = \alpha_{\OO_{\mathcal X}(D), 2}(\shf F)-1 $.
	It follows that if we set 
	\[\delta_1\ge {(C-\hat{\mu})\big((\alpha_{\OO_{\mathcal X}(D), 2}(\shf F)-1\big)\alpha_{\OO_{\mathcal X}(D), 2}(\shf F))},\] condition 2 of \cref{def:muHatStability} will always hold.
	
	Arguing similarly for condition 1) of \cref{def:muHatStability}, we find that it will hold if \[ \delta_1 \ge (C-\hat{\mu})(\alpha_{\OO_{\mathcal X}(D), 2}(\shf F)).\]
	It is thus sufficient to pick
	\[ \delta_1\ge \max{\left\{{(C-\hat{\mu})\big((\alpha_{\OO_{\mathcal X}(D), 2}(\shf F)-1\big)\alpha_{\OO_{\mathcal X}(D), 2}(\shf F))},\ (C-\hat{\mu})(\alpha_{\OO_{\mathcal X}(D), 2}(\shf F))\right\}} .\]

\end{proof}
It follows that there is a $ \delta(\nu)\in \Q[\nu] $ such that the sheaves in $ Y_{\mathbf r,n} $ are stable with respect to $ \delta $.

We are finally ready to prove \cref{thm:main}:

\begin{proof}[Proof of \cref{thm:main}]
	 Let $ P $ be the Hilbert polynomial common to all elements of $ Y_{\mathbf r,n} $, which we know exists by \cref{cor:sameHilbPoly}, and let $ \delta(\nu) =\delta_1\nu+\delta_0$ be a linear polynomial such that all elements of  $ Y_{\mathbf r,n} $ are $ \hat{\mu} $-stable with respect to $ \delta_1 $. By  \cref{lem:sheavesInYrnAreStable},such a $ \delta $ exists. Then, by \cref{thm:ModSpaceExists}, there exists a quasiprojective scheme \[{\mathbf M}_{\mathcal X}(\OO_{\mathcal X}(D), \OO_{\P^2/\Gamma}(1), P, i_*\shf R, \delta(\nu))\] parametrising $ i_*\shf R $-framed sheaves $ (\shf F, \phi_{\shf F})$ that are $ \delta(\nu) $-stable, and have modified Hilbert polynomial equal to $ P $. 
	 It is clear that there is a subscheme \[  \mathbf{Y}_{\mathbf r,n} \subset{\mathbf M}_{\mathcal X}(\OO_{\mathcal X}(D), \OO_{\P^2/\Gamma}(1), P, i_*\shf R, \delta(\nu))\] representing  the functor from \cref{thm:main}.
	 
	Now, given an $i_*\shf R $-framed sheaf $ (\shf F, \phi_{\shf F})$, the framing morphism $ \phi_{\shf F}\colon \shf F\to i_*\shf{R} $ may not necessarily be the adjoint of an \emph{isomorphism} $  i^*\shf F=\shf F|_{L}\isoto \shf R $. Furthermore, $ \shf F $ may not be torsion-free. It is, however, easy to see that if $ \phi_{\shf F} $ is the adjoint of an isomorphism, the isomorphism class of the pair $ (\shf F, \phi_{\shf F})$ lies in $ \mathbf Y_{\mathbf r,n} $.	So to show that $ {\mathbf{Y}_{{\mathbf r},n}} $ is quasiprojective, we only need to show that the locus $ \mathbf{Y}_{\mathbf r,n}\subset {\mathbf M}_{\mathcal X}(\OO_{\mathcal X}(D), \OO_{\P^2/\Gamma}(1), P, i_*\shf{R}, \delta(\nu)) $ is open.

	There are two parts to this; we will show that
	\begin{itemize}
		\item the torsion-freeness of $ (\shf F, \phi_{\shf F})\in {\mathbf M}_{\mathcal X}(\OO_{\mathcal X}(D), \OO_{\P^2/\Gamma}(1), P, i_*\shf R, \delta(\nu))$ is an open condition;
		\item that the framing morphism $\phi_{\shf F}$ should restrict to an isomorphism on $D$ is an open condition.
	\end{itemize}
	For torsion-freeness, this is simply \cite[Proposition 2.31]{BruzzoSala}. (Had we been working with schemes instead of projective stacks, this would have been \cite[Proposition 2.1]{Maruyama})
	
	As for the second point, every flat family of framed sheaves $ \{(\shf F, \phi_{\shf F}\colon i^*\shf F\to \shf R)\} $ on $ \mathcal{X} $ restricts to a family of morphisms $ \{\phi_{\shf F}\colon i^*\shf F\to \shf R\} $ on $ D $. In such a family, the condition $\coker \phi_{\shf F}=0  $ is open by semicontinuity, and on the locus where $ \phi_{\shf F} $ is surjective, $ \ker \phi_{\shf F}=0 $ is similarly an open condition.
	
\end{proof}
%

\begin{corollary}
	The functor $ D $ from \cref{def:UpDownFunctors} induces a scheme morphism
	\[ \mathbf X_{\mathbf r,\mathbf v}\to \mathbf Y_{\mathbf r,n} \] for any $ \mathbf v\in \Z_{\ge 0}^{Q_{\Gamma, 0}} $ such that $ \mathbf v_0=n $.
\end{corollary}
\begin{proof}
	It is enough to show that applying $D$ to the universal family $ \mathcal U $ on $ \mathbf X_{\mathbf r,\mathbf v} $ induces a flat family of sheaves on $ \mathbf Y_{\mathbf r,n} $. 
	So $ \mathcal U $ is a sheaf on $ \P^2\times \mathbf X_{\mathbf r,\mathbf v} $. Since $ D $ works relatively, we have a sheaf $ D(\mathcal U) $ on $ \mathcal X\times \mathbf X_{\mathbf r,\mathbf v} $, and we only need to show that it is flat over $ \mathbf X_{\mathbf r,\mathbf v} $. To see this, we show that $ D(\mathcal U)|_{V_D\times \mathbf X_{\mathbf r,\mathbf v}} $ and $ D(\mathcal U)|_{V_p\times \mathbf X_{\mathbf r,\mathbf v}}  $ are flat over $ \mathbf X_{\mathbf r,\mathbf v} $. The first is clear, while the second reduces to showing the following claim: 
\emph{
	If $ M $ be a $ \Gamma $-equivariant $\C[x,y]\otimes R $-module (for some ring $ R $) which is flat as a $ R $-module, then $ M^\Gamma $ is a $ \C[x,y]^\Gamma\otimes R $-module flat as an $ R $-module.}

	To show the claim, let $ 0\to K \to E \to F\to 0 $ be a short exact sequence of $  R $-modules. Then there is a commutative diagram with exact rows
	\[\begin{tikzcd}
	& M^\Gamma\otimes K \arrow[d] \arrow[r, "j"] & M^\Gamma\otimes E \arrow[d] \arrow[r] & M^\Gamma\otimes F \arrow[d] \arrow[r] & 0 \\
	0 \arrow[r] & M\otimes K \arrow[r]                  & M\otimes E \arrow[r]                  & M\otimes F \arrow[r]                  & 0
	\end{tikzcd}.\]
	
	Here the downward homomorphisms are injective because $ M^\Gamma $ is a direct summand of $ M $. This shows that $ j $ is injective, so $ M^\Gamma $ is flat.

\end{proof}

Finally, we show that the schemes $\mathbf Y_{\mathbf r, n}$ indeed generalise the punctual Hilbert scheme $\Hilb^n(\C^2/\Gamma)$, as mentioned in the Introduction. Let $\mathbf{1}$ be the vector $(1,0,\dots, 0)$.
\begin{corollary}\label{cor:rank1Example}
	There is an isomorphism $ \mathbf Y_{\mathbf 1,n}\isoto \Hilb^n(\C^2/\Gamma) $.
\end{corollary}
\begin{proof}
Let $ \shf F\in Y_{\mathbf 1,n} $, we show that $ \shf F|_{U_p} \in \Hilb^n(\C^2/\Gamma)$.
	Then $ \pi^T\shf F $ is an element of some $ X_{1, \mathbf v} $ with $ \mathbf v_0=n $, especially it is by \cite[Proposition 2.8]{Nakajimabook} an ideal sheaf on $ \P^2 $.
	It is easy to show that $D(\OO_{\P^2}) = \OO_{\mathcal{X}}$. Now $ D(\pi^T\shf F)\cong \shf F $ by \cref{lem:upanddowniso}, by which it follows (since $D$ is exact) that $ \shf F $ is an ideal sheaf on $ \mathcal X $.
	We can tensor the ideal sheaf sequence for $ \shf F $ with $ \shf I_{D} $, obtaining  \[ 0\to \shf F\otimes \shf I_{D}\to \shf I_{D}\to \OO_{\mathcal{X}}/\shf F\otimes \shf I_{D}\to 0 \] which is a short exact sequence because $ \shf I_{D} $ is locally free. Because $ \shf F $ is trivial along $ D $, it is the ideal sheaf of a zero-dimensional subscheme not intersecting $ D $, and so we have $ \OO_{\mathcal{X}}/\shf F\otimes \shf I_{D} \cong \OO_{\mathcal{X}}/\shf F$. Taking the long exact sequence of cohomology, we find that \[ h^1(\mathcal{X}, \shf F\otimes \shf I_{D}) =h^0(\OO_{\mathcal{X}}/\shf F)=n,\] so $ \shf F $ is indeed a point of $ \Hilb^n(\C^2/\Gamma) $.
	
	In the other direction, take a colength $ n $ zero-dimensional subscheme $Z\subset \C^2/\Gamma $, and let $ \shf I_Z $ be the ideal sheaf of $ Z $ as a substack of $ \mathcal X $. It is straightforward to check that $ \shf I_Z\in Y_{\mathbf 1,n} $, and this gives the inverse map to the previous construction.
	Both these constructions work relatively, and so we get an isomorphism $ \mathbf Y_{\mathbf 1, n}\isoto \Hilb^n(\C^2/\Gamma) $.
\end{proof}

\begin{remark}\label{rmk:quiverbij}
	In \cite{Paper1}, we showed that  given two positive integers $(r,n)$, there is a {Nakajima quiver variety} $\mathfrak M_{\theta_0}(r, n\delta)$, the closed points of which parametrise $(D, \OO_D^{\oplus r})$-framed coherent torsion-free sheaves of rank $r$ on $\mathcal{X}$. It follows that, if we set $\mathbf r = (r, 0,\dots,0)$, there is a bijection on the level of closed points $\mathbf Y_{\mathbf r, n}(\C)\isoto \mathfrak M_{\theta_0}(r, n\delta)(\C)$. We conjecture that this bijection can be extended to an isomorphism of schemes -- or at least, to extend the statement in \cite[Corollary 6.3]{CGGS}, of the underlying reduced schemes -- but we remain unable to prove this.
	
\end{remark}

\renewcommand{\bibname}{References}

\printbibliography
	\end{document}